\documentclass[a4paper,10pt]{amsart}
\usepackage[utf8]{inputenc}
\usepackage{lmodern}
\usepackage[T1]{fontenc}
\usepackage{microtype} 	
\usepackage{amsmath,amssymb,amsfonts,amsthm,latexsym,mathrsfs}
\usepackage{color}
\usepackage{hyperref}
\hypersetup{colorlinks=true,linkcolor=red,citecolor=blue}
\usepackage{multirow}
\usepackage{longtable}
\usepackage{graphicx}
\usepackage[all]{xy}
\input xypic
\theoremstyle{plain}
\newtheorem{theorem}{{\bf Theorem}}
\newtheorem*{theorem*}{{\bf Theorem}}
\newtheorem{corollary}[theorem]{{\bf Corollary}}
\newtheorem*{corollary*}{{\bf Corollary}}

\newtheorem{lemma}[theorem]{{\bf Lemma}}

\newtheorem{question}[theorem]{{\bf Question}}
\newtheorem{conjecture}[theorem]{{\bf Conjecture}}
\theoremstyle{definition}

\theoremstyle{remark}

\newtheorem*{example}{{\it Example}}
\DeclareMathOperator{\im}{im}
\DeclareMathOperator{\kk}{k}
\DeclareMathOperator{\Wh}{Wh}

\DeclareMathOperator{\SK}{SK}
\DeclareMathOperator{\FF}{\mathbb F}
\DeclareMathOperator{\ab}{ab}
\DeclareMathOperator{\I}{I}
\DeclareMathOperator{\Log}{Log}
\DeclareMathOperator{\Hom}{Hom}

\DeclareMathOperator{\rad}{Rad}		
\DeclareMathOperator{\gr}{gr}		
\DeclareMathOperator{\B}{B}			
\DeclareMathOperator{\K}{K}			
\DeclareMathOperator{\GL}{GL}		
\DeclareMathOperator{\trf}{trf}		
\DeclareMathOperator{\incl}{incl}	
\DeclareMathOperator{\id}{id}		
\DeclareMathOperator{\Gal}{Gal}		
\DeclareMathOperator{\R}{R}			
\DeclareMathOperator{\Q}{Q}			
\DeclareMathOperator{\Aut}{Aut}		
\DeclareMathOperator{\coker}{coker}	
\DeclareMathOperator{\M}{M}	
\DeclareMathOperator{\HH}{H}	
\DeclareMathOperator{\stab}{Stab}	
\DeclareMathOperator{\f}{F}			
\newcommand{\ZZ}{\mathbb{Z}}
\newcommand{\QQ}{\mathbb{Q}}

\newcommand{\Fp}{\mathbb{F}_{p}}    
\newcommand{\Fq}{\mathbb{F}_{q}}	
\newcommand{\Fl}{\mathbb{F}_{l}}	
\newcommand{\F}{\mathbb{F}}			
\begin{document}
\baselineskip=14pt
\title[Units of group rings]{Units of group rings, the Bogomolov multiplier, and the fake degree conjecture}
\author[J. Garc\'ia-Rodr\' iguez]{Javier Garc\'ia-Rodr\' iguez }
\address{
Javier Garc\'ia-Rodr\' iguez \\
Departamento de Matem\'aticas, Universidad Aut\'onoma de Madrid and Instituto de Ciencias Matem\'aticas \\
Madrid \\
Spain}
\email{javier.garciarodriguez@uam.es}
\author[A. Jaikin-Zapirain]{Andrei Jaikin-Zapirain}
\address{
Andrei Jaikin-Zapirain \\
Departamento de Matem\'aticas, Universidad Aut\'onoma de Madrid and Instituto de Ciencias Matem\'aticas \\
Madrid \\
Spain}
\email{andrei.jaikin@uam.es}
\author[U. Jezernik]{Urban Jezernik}
\address{
Urban Jezernik \\
Institute of Mathematics, Physics, and Mechanics \\
Ljubljana \\
Slovenia}
\email{urban.jezernik@imfm.si}

\date{\today}
\begin{abstract}
Let $\pi$ be a finite $p$-group and $\FF_q$ a finite field with $q=p^n$
elements. Denote by $\I_{\FF_q}$  the augmentation ideal of the group ring
$\Fq[\pi]$. We have found a surprising relation between the abelianization of
$1+\I_{\Fq}$, the Bogomolov multiplier $\B_0(\pi)$ of $\pi$ and the number of
conjugacy classes $\kk(\pi)$ of $\pi$:
\[
\left | (1+\I_{\Fq})_{\ab} \right |=q^{\kk(\pi)-1}|\!\B_0(\pi)|.
\]
In particular, if $\pi$ is a finite $p$-group with a non-trivial Bogomolov
multiplier, then $1+\I_{\Fq}$ is a counterexample to the fake degree conjecture
proposed by M. Isaacs.
\end{abstract}
\maketitle
\section{Introduction}

\noindent Let $J$ be a finite dimensional nilpotent algebra over a finite  field $\f$.
Then the set $G=1+J$ is a finite group. The groups constructed in this way are
called {\em algebra groups}.  The group $G$ acts by conjugation on $J$.  This
induces an action of $G$ on the dual space $J^*=\Hom_{\f}(J,\f)$.  It has been noted
that there exists a relation between the characters of $G$ and the orbits of
$J^*$. For example, if $J^p=0$, there exists an explicit expression that gives a
bijective correspondence between the characters of $G$ and the orbits of $J^*$
(\cite{Sa01}). In particular, when $J^p=0$, we  obtain that the character degrees
of $G$, counting  multiplicities, are the square roots of the sizes of the
$G$-orbits in $J^*$. It was conjectured by M. Isaacs that the same holds also in
the general case:

\begin{conjecture}[Fake degree conjecture]
\label{fd}
In every algebra group $G=1+J$ the character degrees coincide, counting
multiplicities, with the square roots of the cardinals of the $G$-orbits in
$J^*$.
\end{conjecture}

Note that an immediate corollary of this conjecture (see Lemma
\ref{l:abelianization}) is that the orders of $[J,J]_L$ and $[1+J,1+J ]_G$ have
to be equal (in this work we write $[a ,b ]=a^{-1}b^{-1}ab$ for group
commutators and $[a ,b]_L=ab-ba$ for Lie brackets). Thus in order to understand
Conjecture \ref{fd}, one should first answer the following question.

\begin{question}
\label{q:abel}
Is it true that the size of the abelianization of $1+J$ coincides with the index
of $[J,J]_L$ in $J$?
\end{question}

In \cite{Ja1} an example that provides a negative answer to Question \ref{q:abel} in
characteristic $2$ was constructed. However, in questions related to character
correspondences for finite $p$-groups the prime $p=2$  always plays a special
role (see, for example, \cite{Ja2}), and so one might hope that Conjecture
\ref{fd} still holds in odd characteristic.

This was our motivation for looking at the following family of
examples. Let $\pi$ be a finite $p$-group. Given a ring $R$ we will set $\I_{R}$
to be the augmentation ideal of the group ring $R[\pi]$. If we take $R=\Fq$,
then $\I_{\Fq}$ is a nilpotent algebra and  $1 + \I_{\FF_q}$ is the group of
normalized units of the modular group ring $\Fq[\pi]$. It is not difficult to
see that the index of $[\I_{\FF_q}, \I_{\FF_q}]_L$ in $\I_{\FF_q}$ is equal to
$q^{\kk(\pi)-1}$, where $\kk(\pi)$ is the number of conjugacy classes of $\pi$
(see Lemma \ref{l:liealgebra}). Our main result describes the size of the
abelianization $(1+\I_{\FF_q})_{\ab}$ of $1 + \I_{\FF_q}$.

\begin{theorem}\label{th:sizeequality}
Let $\pi$ be a finite $p$-group. Then
$|(1 + \I_{\FF_q})_{\ab}| = q^{\kk(\pi) - 1} |\!\B_0(\pi)|$.
\end{theorem}

The group $\B_0(\pi)$ that appears in the theorem is the {\em Bogomolov
multiplier} of $\pi$. It is defined as the subgroup of the Schur multiplier
$\HH^2(\pi,\QQ/\ZZ)$ of $\pi$ consisting of the cohomology classes vanishing
after restriction to all abelian subgroups of $\pi$.  The Bogomolov multiplier
plays an important role in birational geometry of quotient spaces $V/\pi$ as it
was shown by Bogomolov in \cite{Bog88}. In a dual manner, one may view the group
$\B_0(\pi)$ as an appropriate quotient of the homological Schur multiplier
$\HH_2(\pi,\ZZ)$, see \cite{Mor12}. We were surprised to discover that, in this
form, the Bogomolov multiplier had appeared in the literature much earlier in a
paper of W. D. Neumann \cite{Ne}, as well as in the paper of B. Oliver
\cite{Oli80} that plays an essential role in our proofs. The latter paper
contains various results about Bogomolov multipliers that were only subsequently
proved in the cohomological framework.

There are plenty of finite $p$-groups with non-trivial Bogomolov multipliers
(see, for example, \cite{Kan14}). Thus we obtain a negative solution to the fake
degree conjecture for all primes.

\begin{corollary}
\label{c:examples}
For every prime $p$ there exists a finite dimensional nilpotent $\F_p$-algebra
$J$ such that the size of the abelianization of $1+J$ is greater than the index
of $[J,J]_L$ in $J$. In particular, the fake degree conjecture is not valid in
any characteristic.
\end{corollary}

Our next result provides a conceptual explanation for the equality in Theorem
\ref{th:sizeequality}. Let $\FF$ be an algebraic closure of $\FF_p$. One can
think of $\mathbf G=1+\I_{\FF}$ as an algebraic group defined over $\FF_p$.
It is clear that $\mathbf G$ is a unipotent group. A direct calculation shows
that  the   Lie algebra $\mathfrak L(\mathbf G)$ of $\mathbf G$ is isomorphic to
$\I_{\FF}$.  We write $\mathbf G(\Fq)$ for the $\Fq$-points of $\mathbf G$.
The derived subgroup $\mathbf G^\prime$ of $\mathbf G$ is also a unipotent
algebraic group defined over $\F_p$ (see \cite[Corollary I.2.3]{Bor91}), and so by
\cite[Remark A.3]{KMT74}, $|\mathbf G^\prime(\F_q)|=q^{\dim \mathbf G^\prime}$.
Note that in general we have only an inclusion $$(1+\I_{\FF_q})^\prime=(\mathbf
G(\Fq))^\prime \subseteq \mathbf G^\prime(\Fq),$$ but not the equality.

\begin{theorem} 
\label{t:algebraicgroup} Let $\pi$ be a finite $p$-group and $\mathbf G=1+\I_{\F}$. 
\begin{enumerate}
\item
We have
\[
\dim \mathbf G^\prime=\dim_{\FF}[\mathfrak L(\mathbf G),\mathfrak L(\mathbf G)  ]_L=|\pi|-\kk(\pi).
\]
In particular,
\[
|\mathbf G(\FF_q):\mathbf G^\prime(\FF_q)|=q^{\kk(\pi)-1}.
\]

\item  For every $q=p^{n}$, we have
\[
\mathbf G'(\Fq)/\mathbf G(\Fq)'\cong \B_{0}(\pi).
\]
\end{enumerate}

\end{theorem}

Our hope is that the second statement of the theorem would help better
understand the structure of the Bogomolov multiplier. As an example of this
reasoning, recall that a classical problem about the Schur multiplier asks what
is the relation between the exponent of a finite group and of its Schur
multiplier (\cite{Sch04}). Standard arguments reduce this question to the case
of $p$-groups.  It is known that the exponent of the Schur multiplier is bounded
by some function that depends only on the exponent of the group (\cite{Mor07}),
but this bound  is obtained from  the bounds that appear in the solution of the
Restricted Burnside Problem and so it is probably very far from being optimal.
Applying to the homological description of the Bogomolov multiplier,  it is not
difficult to see that the exponent of the Schur multiplier is at most the
product of the exponent of the group by the exponent of the Bogomolov
multiplier. Thus, we hope that the following theorem would help obtain a
better bound on the exponent of the Schur multiplier.

 \begin{theorem}\label{exponent}
Let $\pi$ be a finite $p$-group and $\mathbf G=1+\I_{\F}$. For every $q=p^{n}$, we have
\[
\exp \B_0(\pi) = \min \{ m \mid \mathbf G^\prime (\FF_q) \subseteq \mathbf G(\FF_{ q^{m} })^\prime \}.
\]
 \end{theorem}

\noindent {\em Acknowledgments:} Javier Garc\'ia-Rodr\'iguez and Andrei Jaikin-Zapirain partially supported by the   grant MTM 2011-28229-C02-01 of the Spanish MEyC and by the   ICMAT Severo Ochoa project SEV-2011-0087. The first author also supported by the FPI grant BES-2012-051797 of the Spanish MINECO. Urban Jezernik supported by the Slovenian Research Agency and in part by the Slovene Human Resources Development and Scholarship Fund.

This research was in part carried out while the third author was
visiting the Universidad Aut\'onoma de Madrid and Instituto de Ciencias Matem\'aticas.
He would like to thank everyone involved for their fine hospitality.

\section{Proofs of the results}
\subsection{Proof of Theorem \ref{th:sizeequality}} 
Let $\zeta_n$ denote a primitive $n$-root of unity. If $p$ is a prime and $q $ is a power of $p$, let
$\R_q=\ZZ_p[\zeta_{q-1}]$ be a finite extension of the $p$-adic integers
$\ZZ_p$. Note that   $\R_q/p\R_q \cong \FF_q$.
Fix a $\ZZ_p$-basis $ B_{q}=\{
\lambda_j \mid 1 \leq j \leq n \}$ of $\R_q$ and let $\varphi$ be a generator of
$\Aut(\R_q|\ZZ_p) \cong \Gal(\FF_q|\FF_p)$ such that $\varphi(\lambda) \cong
\lambda^p \pmod{p}$.
Let us define
\[
\bar\I_{\R_q} = \I_{\R_q}/\langle x - x^g \mid x \in \I_{\R_q}, \, g \in \pi
\rangle.
\] 
Set $\mathcal C$ to be a set of nontrivial conjugacy class representatives of
$\pi$. Then $\bar\I_{\R_q}$ can be regarded as a free $\ZZ_p$-module with basis
$\{ \lambda \overline{(1 - r)} \mid \lambda\in B_{q}, \, r \in \mathcal C \}$.
Finally define the abelian group $\M_q$ to be
\[
\M_q=\bar\I_{\R_q}/\langle p \lambda \overline{(1 - r)} - \varphi(\lambda)\overline{(1 - r^p)} \mid \lambda\in B_{q},\ r \in \mathcal C\rangle.
\]
The proof of Theorem \ref{th:sizeequality} rests on the following structural
description of the group $(1 + \I_{\FF_q})_{\ab}$.

\begin{theorem}
\label{th:exactseq}
Let $\pi$ be a finite $p$-group. There is an exact sequence
\begin{equation*}
\label{eq:extheorem}
\xymatrix{ 
1 \ar[r] & \B_0(\pi) \times \pi_{\ab} \ar[r] & (1 + \I_{\FF_q})_{\ab} \ar[r] & \displaystyle \M_q \ar[r] & \pi_{\ab} \ar[r] & 1.
}
\end{equation*}
\end{theorem}
\begin{proof}
Given a ring $\R$, recall the first $\K$-theoretical group 
$\K_{1}(\R)= \GL(\R)_{\ab}$.
When $\R$ is a local ring, there is an isomorphism
$\K_{1}(\R)\cong \R^*_{\ab}$ (see \cite[Corollary 2.2.6]{Ros94}). 
We therefore have $\K_1(\FF_q[\pi]) \cong \FF_q^* \times (1 + \I_{\FF_q})_{\ab}$, and
our proof relies on inspecting the connection between $\K_1(\FF_q[\pi])$ and
$\K_1(\R_q[\pi])$ by utilizing the results of \cite{Oli80}.

Put $\Q_q$ to be the ring of fractions of $\R_q$ and let
\[
\SK_1(\R_q[\pi]) = \ker\left( \K_1( \R_q[\pi] ) \to \K_1(\Q_q[\pi]) \right).
\]
By \cite[Theorem 3]{Oli80}, we have that $\SK_1(\R_q[\pi]) \cong \B_0(\pi)$.
Now set
\[
\Wh'(\R_q[\pi]) = \K_1(\R_q[\pi])/(\R_q^* \times \pi_{\ab} \times \SK_1(\R_q[\pi])).
\]
The crux of understanding the structure of the group $\K_1(\R_q[\pi])$ is in the
short exact sequence (see \cite[Theorem 2]{Oli80})
\begin{equation*}
\xymatrix{ 
1 \ar[r] & \Wh'(\R_q[\pi]) \ar[r]^-{\Gamma} & \bar\I_{\R_q} \ar[r] & \pi_{\ab} \ar[r] & 1,
}
\end{equation*}
where the map $\Gamma$ is defined by composing the $p$-adic logarithm with a
linear automorphism of $\bar\I_{\R_q} \otimes \QQ_p$. More precisely, there is a
map $\Log \colon 1 + \I_{\R_q} \to \I_{\R_q} \otimes \QQ_p$, which induces an
injection $\log \colon \Wh'(\R_q[\pi]) \to \bar\I_{\R_q} \otimes \QQ_p$. Setting
$\Phi \colon \I_{\R_q} \to \I_{\R_q}$ to be the map  $\sum_{g \in \pi} \alpha_g
g \mapsto \sum_{g \in \pi} \varphi(\alpha_g) g^p$, we define $\Gamma \colon
\Wh'(\R_q[\pi]) \to \bar\I_{\R_q} \otimes \QQ_p$ as the  composite of $\log$
followed by the linear map $1 - \frac{1}{p} \Phi$. It is shown in
\cite[Proposition 10]{Oli80} that  $\im \Gamma \subseteq \bar\I_{\R_q}$, i.e.,
$\Gamma$ is integer-valued. We thus have a diagram
\begin{equation}
\label{eq:diagWh}
\xymatrix{ 
 & 1 + \I_{R_q} \ar[d] \ar@/^/[dr]^-{(1 - \frac{1}{p}\Phi)\circ \log} \\
1 \ar[r] & \Wh'(\R_q[\pi]) \ar[r]^-{\Gamma} & \bar\I_{R_q}.
}
\end{equation}
The group $\Wh'(\R_q[\pi])$ is torsion-free (cf. \cite {Wa}), so we have an
explicit description
\begin{equation}
\label{eq:k1explicit}
\K_1(\R_q[\pi]) \cong  \R_q^*\times\SK_1(\R_q[\pi]) \times \pi_{\ab} \times \Wh'(\R_q[\pi]).
\end{equation}

To relate the above results to $\K_1(\FF_q[\pi])$, we invoke a part of the
$\K$-theoretical long exact sequence for the ring $\R_q[\pi]$ with respect to
the ideal generated by $p$,
\begin{equation}
\label{eq:exseqK}
\xymatrix{ 
\K_1(\R_q[\pi], p) \ar[r]^-{\partial} & \K_1(\R_q[\pi]) \ar[r]^-{\mu} & \K_1(\FF_q[\pi]) \ar[r] & 1.
}
\end{equation}
Note that $\K_1(\R_q[\pi], p) = (1 + p\R_q) \times \K_1(\R_q[\pi],p\I_{\R_q})$
and $\R_q^*/(1 + p\R_q) \cong \FF_q^* $. 
Hence \eqref{eq:k1explicit} and \eqref{eq:exseqK} give a reduced exact sequence
\begin{equation}
\label{eq:exseqKreduced}
\xymatrix{ 
\K_1(\R_q[\pi], p\I_{\R_q}) \ar[r]^-{\partial} & \Wh'(\R_q[\pi]) \ar[r]^-{\mu} & \displaystyle \frac{(1 + \I_{\FF_q})_{\ab}}{\mu(\SK_1(\R_q[\pi]) \times \pi_{\ab})} \ar[r] & 1.
}
\end{equation}
To determine the structure of the relative group $\K_1(\R_q[\pi], p\I_{\R_q})$
and its connection to the map $\partial$, we make use of \cite[Proposition
2]{Oli80}. The restriction of the logarithm map $\Log$ to $1 + p\I_{\R_q}$
induces an isomorphism $\log \colon \K_1(\R_q[\pi], p\I_{\R_q}) \to
p\bar\I_{\R_q}$ such that the following diagram commutes:
\begin{equation}
\label{eq:diagK1rel}
\xymatrix{ 
 & 1 + p\I_{\R_q} \ar[d] \ar@/^/[dr]^-{\log} \\
1 \ar[r] & \K_1(\R_q[\pi],p\I_{\R_q}) \ar[r]^-{\log} & p\bar\I_{\R_q}.
}
\end{equation}
In particular, the group $\K_1(\R_q[\pi],p\I_{\R_q})$ is torsion-free, and so
$\mu(\SK_1(\R_q[\pi]) \times \pi_{\ab}) \cong \SK_1(\R_q[\pi]) \times
\pi_{\ab}$. Note that by \cite[Theorem V.9.1]{Bas68}, the vertical map $1 +
p\I_{\R_q} \to \K_1(\R_q[\pi],p\I_{\R_q})$ of the above diagram is surjective.

We now collect the stated results to prove the theorem. First combine the
diagrams \eqref{eq:diagWh} and \eqref{eq:diagK1rel} into the following diagram:
\begin{equation}
\label{eq:diagCompare}
\xymatrix{ 
1 + p\I_{\R_q} \ar[d] \ar[rr] \ar@/^/[ddr]^{\log} && 1 + \I_{\R_q} \ar[d] \ar@/^/[ddr]^{(1 - \frac{1}{p}\Phi)\circ \log} \\
\K_1(\R_q[\pi],p\I_{\R_q}) \ar[rr]^-{\partial} \ar[dr]^-{\log} && \Wh'(\R_q[\pi]) \ar[dr]^-{\Gamma} \\
& p\bar\I_{\R_q} \ar[rr]^-{1 - \frac{1}{p}\Phi} && \bar\I_{\R_q}.
}
\end{equation}
Since the back and top rectangles commute and the left-most vertical map is
surjective, it follows that the bottom rectangle also commutes.  Whence $\coker
\partial \cong \coker (1 - \frac{1}{p}\Phi)$. Observing that the latter group is
isomorphic to $\M_q$, the exact sequence \eqref{eq:exseqKreduced} gives an exact
sequence
\begin{equation}
\label{eq:exfinal}
\xymatrix{ 
1 \ar[r] & \B_0(\pi) \times \pi_{\ab} \ar[r] & (1 + \I_{\FF_q})_{\ab} \ar[r] & \M_q \ar[r] & \pi_{\ab} \ar[r] & 1.
}
\end{equation}
The proof is complete.
\end{proof}

We now derive Theorem \ref{th:sizeequality} from Theorem \ref{th:exactseq}.

\begin{proof}[Proof of Theorem \ref{th:sizeequality}]
The exact sequence of Theorem \ref{th:exactseq} implies that $|(1 +
\I_{\FF_q})_{\ab}| = |\!\B_0(\pi)| \cdot |\!\M_q\!|$. Hence it suffices to
compute $|\!\M_q\!|$. To this end, we filter $\M_q$ by the series of its
subgroups
\[
\M_q\supseteq p\M_q\supseteq p^2\M_q\supseteq\cdots\supseteq p^{\log_{p}(\exp\pi)}\M_q.
\]
Note that the relations $p \lambda \overline{(1 - r)} -
\varphi(\lambda)\overline{(1 - r^p)}=0$ imply $p^{\log_{p}(\exp\pi)}\M_q=0$.

For each $0\leq i\leq \log_p(\exp\pi)$, put
\[
\pi_i=\{x^{p^{i}}\mid x\in\pi\} \textrm{\ and \ }  \mathcal{C}_{i}=\mathcal{C}\cap(\pi_i\setminus \pi_{i+1}).
\]
Then
\[
 p^{i}\M_q/p^{i+1}\M_q=\langle\lambda \overline{(1-r)}:\lambda\in B_{q},\ r\in\mathcal{C}_i\rangle \cong \bigoplus_{\mathcal{C}_i} C_p^{n}. 
\]
It follows that $|\!\M_q\!|=q^{|\mathcal{C}|}=q^{\kk(\pi)-1}$ and the proof is complete.
\end{proof}

\begin{example}
Let $\pi$ be the group given by the polycyclic generators $\{ g_i \mid 1 \leq i
\leq 7 \}$ subject to the power-commutator relations
\[
\begin{aligned}
g_{1}^{2} = g_{4}, \,
 g_{2}^{2} = g_{5}, \,
 g_{3}^{2} = g_{4}^{2} = g_{5}^{2} = g_{6}^{2} = g_{7}^{2} = 1, \\
[g_{2}, g_{1}]  = g_{3}, \,
 [g_{3}, g_{1}]  = g_{6} ,\,
 [g_{3}, g_{2}]  = g_{7} , \,
 [g_{4}, g_{2}]  = g_{6}, \,
 [g_{5}, g_{1}]  = g_{7},
\end{aligned}
\]
where the trivial commutator relations have been omitted. The group $\pi$ is of
order $128$ with $\pi_{\ab} \cong C_4 \times C_4$. Its Bogomolov multiplier is
generated by the commutator relation $[g_{3}, g_{2}] = [g_{5}, g_{1}]$ of order
$2$, see \cite[Family 39]{JM14}. We have $\kk(\pi) = 26$ and by inspecting the
power structure of conjugacy classes, we see that $\M_q \cong C_2^{13} \times
C_4^{6}$. On the other hand, using the available computational tool
\cite{LAGUNA}, it is readily verified that we have $(1 + \I_{\FF_q})_{\ab} \cong
C_2^{13} \times C_4^{5} \times C_8$. Following the proof of Theorem
\ref{th:exactseq}, the embedding of $\B_0(\pi) \times \pi_{\ab}$ into $(1 +
\I_{\FF_q})_{\ab}$ maps the generating relation $[g_{3}, g_{2}] = [g_{5},
g_{1}]$ of $\B_0(\pi)$ into the element $\exp((1 - g_7)(g_3 - g_5))$, which
belongs to $(1 + \I_{\FF_q})_{\ab}^4$. In particular, the embedding of
$\B_0(\pi) \times \pi_{\ab}$ into $(1 + \I_{\FF_q})_{\ab}$ may not be split.
\end{example}
\subsection{The fake degree conjecture}
In this subsection we explain in more detail how Corollary \ref{c:examples}
follows from Theorem \ref{th:sizeequality}.

Given an algebra group $G=1+J$ where $J$ is a finite dimensional nilpotent
$\f$-algebra, the fake degree conjecture establishes a bijection between degrees
of irreducible characters of $G$ and the square roots of the lengths of the
coadjoint orbits in $J^*=\Hom_{\f}(J,\f)$. The following result is well known and
enables us to compute lengths of coadjoint orbits. We include its proof for the
reader's convenience.

\begin{lemma}\label{l:stab}
Let $\lambda\in J^{*}$. Define $B_{\lambda}:J\times J\mapsto \f$ to be the
bilinear form which assigns to every pair $(u,v)\in J\times J$ the element
$\lambda([u,v])\in \f$. Then $\stab(\lambda)=1+\rad B_{\lambda}.$
\end{lemma}
\begin{proof}
Let $g=1+u$ be an element of $G$. Then $g$ fixes $\lambda$ if and only if for
every $v\in J$, $\lambda(gvg^{-1})=\lambda(v)$ or equivalently
$\lambda(gvg^{-1}-v)=0$. Since multiplication by $g$ acts bijectively on $J$
this amounts to $\lambda(gv-vg)=\lambda(uv-vu)=\lambda([u,v])=0$ for every $v\in
J$, i.e., $u\in\rad B_{\lambda}$ and the result follows.
\end{proof}

We now focus on $1$-dimensional characters. In this case, the fake degree
conjecture would establish a bijection between linear characters of $G$ and
fixed points of $J^{*}$ under the coadjoint action of $G$.

\begin{lemma}
\label{l:abelianization}
Let $J$ be a finite dimensional nilpotent algebra over a finite field $\f$. Put
$G=1+J$. Then the number of fixed points in $J^*$ under the coadjoint action of
$G$ equals the index of $[J,J]_L$ in $J$. In particular, if the fake degree
conjecture holds, then
\[
|J /[J,J]_{L}|=|(1+J)_{\ab}|.
\]
\end{lemma}  
\begin{proof}
By Lemma \ref{l:stab}, $\lambda\in J^*$ is fixed under the coadjoint action of
$G$ if and only if $\rad B_{\lambda}=J$, which amounts to
$\lambda([J,J]_{L})=0$. The number of fixed points in $J^*$ therefore
equals the number of linear forms vanishing on $[J,J]_{L}$. Hence if the fake
degree conjecture holds, then
\[
|J /[J,J]_{L}|=|\{\text{fixed points of\ } J^{*}\}|=|\{\text{linear characters of\ }G\}|=|(1+J)_{\ab}|. \qedhere
\]
\end{proof}

We now consider the case when $J$ is an augmentation ideal of the
group algebra $\f[\pi]$ of a finite $p$-group $\pi$ over a finite field $\f$ of
characteristic $p$. The ideal $\I_{\f}=\rad \f[\pi]$ is nilpotent
and hence $1+\I_{\f}$ is an algebra group. The following result is well known.

\begin{lemma}
\label{l:liealgebra}
Let $\pi$ be a finite group and $\f$ a field. Then
\[
\dim_{\f} \I_{\f}/[\I_{\f},\I_{\f}]_L= \kk(\pi)-1.
\]
\end{lemma}
\begin{proof}
It is clear that the set $\pi$ is an $\f$-basis for $\f[\pi]$. We first claim
that
\[
\dim_{\f} \f[\pi]/ [\f[\pi],\f[\pi]]_L= \kk(\pi).
\]

Let $x_{1},\ldots,x_{\kk(\pi)}$ be representatives of conjugacy classes of
$\pi$. Observe that for any $x,y,g\in\pi$ with $y=g^{-1}xg$, we have
$x-y=[g,g^{-1}x]_{L}$. The elements
$\bar x_{1},\ldots,\bar x_{\kk(\pi)}$ therefore span $\f[\pi]/
[\f[\pi],\f[\pi]]_L$.

Set $\lambda_{i}$ to be the linear functional on $\f[\pi]$ that takes the value
$1$ on the elements corresponding to the conjugacy class of $x_{i}$ and vanishes
elsewhere. Observe that for any $g,h\in\pi$, we have $[g,h]_{L}=g(hg)g^{-1}-hg$
and hence each $\lambda_{i}$ induces a linear functional on $\f[\pi]/
[\f[\pi],\f[\pi]]_L$. Now if $\sum_j\alpha_{j}\bar x_{j}=0$ for some $\alpha_j
\in \f$, then $\alpha_i = \lambda_{i}(\sum_j \alpha_j \bar x_j)=0$ for each $i$.
It follows that $\bar x_{1},\ldots,\bar x_{\kk(\pi)}$ are also linearly
independent and hence a basis. This proves the claim.


Now, it is clear that $\{g-1:g\in\pi\setminus\{1\}\}$ is an $\f$-basis for
$\I_{\f}$. Since for any $g,h\in\pi$, we have $[g,h]_{L}=[g-1,h-1]_{L}$, it
follows that $[\f[\pi],\f[\pi]]_{L}=[\I_{\f},\I_{\f}]_{L}$, whence the lemma.
\end{proof}

It follows readily from Theorem \ref{q:abel} and Lemma \ref{l:liealgebra} that
whenever $\pi$ is a $p$-group with $\B_0(\pi) \neq 0$, the algebra $J=\I_{\f}$
gives an example for the statement of Corollary \ref{c:examples}. Since for each
prime $p$ there exist groups of order $p^5$ (resp. $2^6$ for $p=2$) with 
non-trivial Bogomolov multipliers (see \cite{HKK13,CHKK10}), Corollary
\ref{c:examples} follows.
\subsection{Proof of Theorem \ref{t:algebraicgroup} and Theorem \ref{exponent}}
\begin{proof}[Proof of Theorem \ref{t:algebraicgroup} and \ref{exponent}]
We will consider an extension $\FF_{l}$ of $\FF_q$ of degree $m$. The inclusion
$\mathbf G(\Fq)\subseteq \mathbf G(\Fl)$ induces a map $f \colon \mathbf G(\Fq)_{\ab}\to \mathbf G(\Fl)_{\ab}$
with
\[
\ker f= (\mathbf G(\Fq)\cap \mathbf G(\Fl)')/\mathbf G(\Fq)'.
\]
Note that there exists a large enough $m$  such that $\mathbf G'(\Fq)= \mathbf G(\Fq)\cap \mathbf G(\Fl)'$, and hence
$\ker f = \mathbf G'(\FF_q)/\mathbf G(\FF_q)'$. For this reason we want to understand $\ker f$
for a given $m$. 


The inclusion $\Fq\subseteq \Fl$ induces a map $$\incl\colon\K_{1}(\Fq[\pi])\to
\K_{1}(\Fl[\pi]).$$ Note that $f$ is just the restriction of $\incl$ to
$(1+\I_{\Fq})_{\ab}$. Recalling sequence \eqref{eq:exseqKreduced} from the proof
of Theorem \ref{th:exactseq}, we set
\[
\SK_{1}(\Fl[\pi])=\mu(\SK_{1}(\R_{l}[\pi])\subseteq(1 + \I_{\Fl[\pi]})_{\ab}=\mathbf G(\Fl)/\mathbf G(\Fl)'.
\]
Commutativity of the diagram
\begin{equation}
\label{eq:diagcomrstr}
\xymatrix{ 
 \K_{1}(\R_{q}[\pi])\ar[r]^-{\mu} \ar[d]^-{\incl} & \K_{1}(\Fq[\pi])\ar[d]^-{\incl}\\
 \K_{1}(\R_{l}[\pi])\ar[r]^-{\mu} & \K_{1}(\Fl[\pi])
 }
\end{equation}
shows that $\incl$ restricts to a map
$\incl\colon\SK_{1}(\Fq[\pi])\to\SK_{1}(\Fl[\pi])$. Recalling that
$\SK_{1}(\R_l[\pi])\cong\B_{0}(\pi)$, we obtain from sequence \eqref{eq:exfinal}
the commutative diagram
\begin{equation*} 
\xymatrix{
1 \ar[r] & \SK_{1}(\Fq[\pi])\times \pi_{\ab} \ar[r]\ar[d]^-{\incl\times \id} & \mathbf G(\Fq)_{\ab} \ar[r]\ar[d]^-{f} & \M_q \ar[r]\ar[d]^-{\iota} & \pi_{\ab} \ar[r] & 1\\
1 \ar[r] & \SK_{1}(\Fl[\pi]) \times \pi_{\ab} \ar[r] & \mathbf G(\Fl)_{\ab} \ar[r] & \M_l \ar[r] & \pi_{\ab} \ar[r] & 1,
}
\end{equation*}
where $\iota$ is the map induced by the inclusion $\I_{\R_{q}}\subseteq \I_{\R_{l}}$.

We will now show that $\ker\iota=0$. This will imply $\ker
f\subseteq\SK_{1}(\Fq[\pi])$. Without loss of generality, we may assume that
there is an inclusion of bases $B_{q}\subseteq B_{l}$. As in the proof of Theorem \ref{th:sizeequality}, let us consider the series
\[
\M_l\supseteq p\M_l\supseteq p^2\M_l\supseteq\ldots\supseteq p^{\log_{p}(\exp\pi)}\M_l.
\]
Observe again that for each $0\leq i\leq\exp\pi-1$ we have
\[
\begin{aligned}
p^{i}\M_q/p^{i+1}\M_q&=\langle\lambda\overline{(1-r)}:\lambda\in B_{q},\ r\in\mathcal{C}_i\rangle,\\
p^{i}\M_l/p^{i+1}\M_l&=\langle\lambda \overline{(1-r)}:\lambda\in B_{l},\ r\in\mathcal{C}_i\rangle.
\end{aligned}
\]
If we consider the graded groups associated to the series above, we get an induced map 
\[
\gr(\iota)\colon \bigoplus_{i=0}^{\exp\pi-1} p^{i}\M_q/p^{i+1}\M_q\to \bigoplus_{i=0}^{\exp\pi-1}p^{i}\M_l/p^{i+1}\M_l.
\]
By construction $\iota$ is induced by the assignments
$\lambda \overline{(1-r_{m})}\mapsto\lambda\overline{(1-r_{m})}$, for every $\lambda\in B_{q},\
r\in\mathcal{C}$. Hence $\gr(\iota)$ is injective in every component and
therefore injective. This implies $\ker\iota=0$, as desired. In particular, we obtain that 
\begin{equation}\label{manyelements}
|\mathbf G(\Fq)/\mathbf G'(\Fq)|\ge|\M_q|= q^{\kk(\pi)-1}.\end{equation}

We are now ready to show the first statement of Theorem \ref{t:algebraicgroup}.
Observe that $\mathbf G$ is a unipotent connected algebraic group defined over $\Fp$ and so is $\mathbf G'$ (\cite[Corollary I.2.3]{Bor91}).
 Hence $\mathbf G'\cong_{\Fp}\mathbb{A}^{\dim \mathbf G'}$ (c.f. \cite[Remark A.3]{KMT74}) and so
$|\mathbf G'(\Fp)|=p^{\dim \mathbf G'}$. By (\ref{manyelements}), 
we have $|\mathbf G(\Fp)/\mathbf G'(\Fp)|\ge p^{\kk(\pi)-1}$, whence $\dim \mathbf G'\le  |\pi|-\kk(\pi)$. 
On the other
hand we have  $ [\mathfrak{L}(G),\mathfrak{L}(G)]_{L}= [\I_{\F},\I_{\F}]_{L},$  which, by Lemma \ref{l:liealgebra}, has dimension $|\pi|-\kk(\pi) $. It is well known that for an algebraic group, $\dim \mathbf G'\geq\dim
[\mathfrak{L}(\mathbf G),\mathfrak{L}(\mathbf G)]_{L}$ (see \cite[Corollary 10.5]{Hum75}). Thus
$ \dim \mathbf G'=|\pi|-\kk(\pi)$.

Let us set $e=\exp\B_{0}(\pi)$. We now claim that $\ker f=\SK_{1}(\Fq[\pi])$ if
and only if $e$ divides $m=|\Fl:\Fq| $. This will imply the second statement of Theorem \ref{t:algebraicgroup} and also  Theorem \ref{exponent}.

Let us consider
$\Fl[\pi]\cong\bigoplus_{i=1}^{m}\Fq[\pi]$ as a free $\Fq[\pi]$-module.  This
gives a natural inclusion $\GL_{1}(\Fl[\pi])\to \GL_{m}(\Fq[\pi])$, which
induces the transfer map
\[
\trf:\K_{1}(\Fl[\pi])\to \K_{1}(\Fq[\pi]).
\]
Note that if $x\in\K_{1}(\Fq[\pi])$, then $(\trf\circ\incl)(x)=x^{m}$.
By commutativity of \eqref{eq:diagcomrstr} the transfer map restricts to a map
\begin{equation*}
\trf:\SK_{1}(\Fl[\pi])\to \SK_{1}(\Fq[\pi]).
\end{equation*}
Moreover, by \cite[Proposition 21]{Oli80} the transfer map is an isomorphism.
It thus follows that $\incl(\SK_{1}(\Fq[\pi]))=1$ if and only if $e$ divides $m$.
Hence $\ker f=\SK_{1}(\Fq[\pi])$ if and only if $e$ divides $m $ and we are done.
 \end{proof}


\end{document}